\documentclass[11pt]{article}

\usepackage[T1]{fontenc}
\usepackage[utf8]{inputenc}
\usepackage{amsmath,amssymb,amsthm}
\usepackage{geometry}
\usepackage{hyperref}

\hypersetup{colorlinks=true,linkcolor=blue,citecolor=blue,urlcolor=blue}

\newtheorem{theorem}{Theorem}
\newtheorem{lemma}{Lemma}

\title{A Median Version of Hardy's Inequality}
\author{Gangsong Leng\\
East China Normal University\\
\texttt{lenggangsong@163.com}}
\date{}

\begin{document}
\maketitle

\begin{abstract}
Motivated by a discrete inequality problem proposed by Duanyang Zhang as Problem 6 of the 2022 Spring NSMO, we prove a median version of Hardy's inequality. For a nonnegative function $f\in L^p(0,\infty)$, $p>1$, let $A(t)$ be the average of $f$ over $(0,t)$, and let $M(t)$ be the lower median of $f$ over $(0,t)$. We show that
\[
 \int_0^\infty |M(t)-A(t)|^p\,dt
 \leq 2^{1-p}\left(\frac p{p-1}\right)^p
 \int_0^\infty f(t)^p\,dt,
\]
and that the constant is best possible. The proof is based on a pointwise rearrangement estimate coming from the half-measure property of the median, followed by the classical Hardy inequality. A discrete form and its sharpness are also included.
\end{abstract}

\medskip


\section{Introduction}

The arithmetic mean and the median are two familiar notions of central tendency. The mean reflects a global balance, while the median reflects the central position after ordering. A natural question is the following: if one considers, for each initial segment, the mean and the median of that segment, can the deviation between them be controlled by a global quantity?

This note is motivated by a problem proposed by Duanyang Zhang as Problem 6 of the 2022 Spring NSMO. The original problem asked for the least constant $\lambda$ such that, for every positive integer $n$ and all nonnegative real numbers $x_1,x_2,\ldots,x_n$,
\[
 \sum_{i=1}^n (m_i-a_i)^2\leq \lambda\sum_{i=1}^n x_i^2,
\]
where
\[
 a_i=\frac{x_1+x_2+\cdots+x_i}{i},
\]
and $m_i$ is a median of $x_1,x_2,\ldots,x_i$. The best constant in this problem is
\[
 \lambda=2.
\]

Behind this discrete problem lies a natural analytic structure. We extend it to a continuous $L^p$ form and obtain a median version of Hardy's inequality. The proof is simple in spirit: the half-measure property of the median gives a rearrangement estimate, and the classical Hardy inequality then yields the desired bound.

\section{From the discrete problem to the continuous problem}

Let $p>1$, and let $f\geq0$ with $f\in L^p(0,\infty)$. For $t>0$, set
\[
 A(t)=\frac1t\int_0^t f(s)\,ds,
\]
the average of $f$ over $(0,t)$. Let $M(t)$ be a median of $f$ on $(0,t)$. To avoid nonuniqueness, we take the lower median, namely
\[
 M(t)=\inf\left\{a\in\mathbb R:
 \operatorname{meas}\{s\in(0,t):f(s)\leq a\}\geq \frac t2\right\}.
\]
Since we consider only nonnegative functions, $M(t)\geq0$. This convention gives a measurable function $t\mapsto M(t)$; this follows, for instance, from the measurability of the distribution function
\[
 (t,a)\mapsto \operatorname{meas}\{s\in(0,t):f(s)\leq a\}.
\]
Hence the integrals involving $M(t)$ below are well defined.

We ask whether there exists a constant $C_p$ such that
\[
 \int_0^\infty |M(t)-A(t)|^p\,dt
 \leq C_p\int_0^\infty f(t)^p\,dt.
\]
The answer is affirmative, and the best constant can be determined exactly.

\section{A brief note on decreasing rearrangements}

Let $f\geq0$ be a measurable function on $(0,\infty)$. Its decreasing rearrangement, denoted by $f^*$, may be understood as the nonincreasing function obtained by rearranging the values of $f$ from largest to smallest.

For finite sequences this idea is completely transparent. Given numbers
\[
 x_1,x_2,\ldots,x_n,
\]
we arrange them in decreasing order as
\[
 y_1\geq y_2\geq\cdots\geq y_n.
\]
Then $y_1,y_2,\ldots,y_n$ are the decreasing rearrangement of the original sequence. The function $f^*$ is the continuous analogue of this operation.

We shall use only the following basic property of decreasing rearrangements: if $E\subset(0,\infty)$ is measurable and has measure $r$, then
\[
 \int_E f(s)\,ds\leq \int_0^r f^*(s)\,ds. \tag{1}
\]
Intuitively, among all sets of measure $r$, the largest possible integral is obtained by taking the part on which $f$ assumes its largest values. For background on decreasing rearrangements and their use in integral inequalities, see Bennett and Sharpley \cite{BS}.

This property is especially well suited to medians. By definition, the portion above the median occupies at most one half, and the portion below the median also occupies at most one half. Thus, when estimating the deviation between the mean and the median, it is enough to control the contribution coming from the largest half of the function values; this is precisely why decreasing rearrangements appear naturally.

\section{The core pointwise estimate}

The following lemma is the key point of the paper. It converts the median problem into a rearrangement estimate.

\begin{lemma}[Median rearrangement estimate]
Let $f\geq0$ be measurable. Define
\[
 A(t)=\frac1t\int_0^t f(s)\,ds,
\]
and let $M(t)$ be the lower median of $f$ on $(0,t)$. Then, for every $t>0$,
\[
 |M(t)-A(t)|\leq \frac1t\int_0^{t/2} f^*(s)\,ds. \tag{2}
\]
Here $f^*$ denotes the decreasing rearrangement of $f$ on the whole interval $(0,\infty)$.
\end{lemma}

\begin{proof}
Fix $t>0$, and write $A=A(t)$ and $M=M(t)$.

First assume that $A\geq M$. Then
\[
 A-M=\frac1t\int_0^t (f(s)-M)\,ds
 \leq \frac1t\int_{\{f>M\}}(f(s)-M)\,ds
 \leq \frac1t\int_{\{f>M\}} f(s)\,ds.
\]
By the definition of the lower median, the set
\[
 \{s\in(0,t):f(s)>M\}
\]
has measure at most $t/2$. Hence, by the rearrangement property (1),
\[
 A-M\leq \frac1t\int_0^{t/2} f^*(s)\,ds.
\]

Now assume that $M\geq A$. We have
\[
 M-A=\frac1t\int_0^t (M-f(s))\,ds
 \leq \frac1t\int_{\{f<M\}}(M-f(s))\,ds
 \leq \frac M t\operatorname{meas}\{s\in(0,t):f(s)<M\}.
\]
Again by the definition of the lower median,
\[
 \operatorname{meas}\{s\in(0,t):f(s)<M\}\leq \frac t2,
\]
and therefore
\[
 M-A\leq \frac M2.
\]
On the other hand,
\[
 \operatorname{meas}\{s\in(0,t):f(s)\geq M\}\geq \frac t2.
\]
Thus, by the defining property of the decreasing rearrangement,
\[
 \int_0^{t/2} f^*(s)\,ds\geq \frac t2 M.
\]
Consequently,
\[
 M-A\leq \frac1t\int_0^{t/2} f^*(s)\,ds.
\]
The two cases together prove (2).
\end{proof}

This lemma is the continuous counterpart of the discrete estimate
\[
 |m_i-a_i|\leq \frac{y_1+y_2+\cdots+y_{\lfloor i/2\rfloor}}{i},
\]
where $y_1\geq y_2\geq\cdots$ is the decreasing rearrangement of $x_1,x_2,\ldots,x_i$. It shows that the deviation between the mean and the median is controlled by the contribution of the largest half of the values.

\section{A median version of Hardy's inequality}

We shall use the classical continuous Hardy inequality: if $p>1$ and $g\geq0$, then
\[
 \int_0^\infty\left(\frac1t\int_0^t g(s)\,ds\right)^p dt
 \leq \left(\frac p{p-1}\right)^p\int_0^\infty g(t)^p\,dt. \tag{3}
\]
The constant $\bigl(p/(p-1)\bigr)^p$ is best possible. For Hardy's inequality and its extensions, see Hardy, Littlewood and P\'olya \cite{HLP}, Kufner, Maligranda and Persson \cite{KMP}, Opic and Kufner \cite{OK}, and Kufner and Persson \cite{KP}; see also Maz'ya \cite{Mazya} for related functional inequalities.

\begin{theorem}[A median version of Hardy's inequality]
Let $p>1$, and let $f\geq0$ with $f\in L^p(0,\infty)$. Define
\[
 A(t)=\frac1t\int_0^t f(s)\,ds,
\]
and let $M(t)$ be the lower median of $f$ on $(0,t)$. Then
\[
 \int_0^\infty |M(t)-A(t)|^p\,dt
 \leq 2^{1-p}\left(\frac p{p-1}\right)^p
 \int_0^\infty f(t)^p\,dt. \tag{4}
\]
Equivalently,
\[
 \|M-A\|_{L^p(0,\infty)}
 \leq 2^{1/p-1}\frac p{p-1}\|f\|_{L^p(0,\infty)}.
\]
\end{theorem}

\begin{proof}
By Lemma 1,
\[
 |M(t)-A(t)|\leq \frac1t\int_0^{t/2} f^*(s)\,ds.
\]
It follows that
\[
 \int_0^\infty |M(t)-A(t)|^p\,dt
 \leq \int_0^\infty\left(\frac1t\int_0^{t/2} f^*(s)\,ds\right)^p dt.
\]
Put $u=t/2$. Then $dt=2\,du$ and
\[
 \frac1t\int_0^{t/2} f^*(s)\,ds
 =\frac12\cdot \frac1u\int_0^u f^*(s)\,ds.
\]
Therefore
\[
 \int_0^\infty\left(\frac1t\int_0^{t/2} f^*(s)\,ds\right)^p dt
 =2^{1-p}\int_0^\infty\left(\frac1u\int_0^u f^*(s)\,ds\right)^p du.
\]
Applying Hardy's inequality (3) to $f^*$, we get
\[
 \int_0^\infty |M(t)-A(t)|^p\,dt
 \leq 2^{1-p}\left(\frac p{p-1}\right)^p
 \int_0^\infty f^*(u)^p\,du.
\]
Since decreasing rearrangement preserves the $L^p$ norm,
\[
 \int_0^\infty f^*(u)^p\,du=\int_0^\infty f(u)^p\,du.
\]
This proves (4).
\end{proof}

When $p=2$, the constant in the theorem becomes
\[
 2^{1-2}\left(\frac21\right)^2=2,
\]
which agrees exactly with the best constant in Zhang's original problem.

\section{Sharpness of the constant}

We next show that the constant in Theorem 1 cannot be improved. The construction is transparent: zero blocks keep the median equal to zero, while the positive blocks approximate extremal behavior for Hardy's inequality.

\begin{theorem}[Sharpness]
The constant
\[
 2^{1-p}\left(\frac p{p-1}\right)^p
\]
in (4) is best possible.
\end{theorem}

\begin{proof}
For a positive integer $N$, define $f_N$ as follows. For $k=1,2,\ldots,N$, let
\[
 f_N(t)=0,\qquad 2k-2\leq t<2k-1,
\]
and
\[
 f_N(t)=k^{-1/p},\qquad 2k-1\leq t<2k.
\]
For $t\geq 2N$, set $f_N(t)=0$.

On every initial interval $(0,t)$, the zero part occupies at least one half of the measure. Therefore, with the lower-median convention,
\[
 M_N(t)=0.
\]
Set
\[
 S_k=\sum_{j=1}^k j^{-1/p}.
\]
If $t\in[2k-1,2k)$, then
\[
 A_N(t)=\frac{S_{k-1}+(t-(2k-1))k^{-1/p}}{t}\geq \frac{S_{k-1}}{2k};
\]
if $t\in[2k,2k+1)$, then
\[
 A_N(t)=\frac{S_k}{t}\geq \frac{S_k}{2k+1}.
\]
Hence
\[
 \int_0^\infty |M_N(t)-A_N(t)|^p\,dt
 \geq \sum_{k=2}^N\left(\frac{S_{k-1}}{2k}\right)^p
 +\sum_{k=1}^{N-1}\left(\frac{S_k}{2k+1}\right)^p.
\]
On the other hand,
\[
 \int_0^\infty f_N(t)^p\,dt=\sum_{k=1}^N\frac1k.
\]
By an elementary integral estimate,
\[
 S_k=\sum_{j=1}^k j^{-1/p}\sim \frac p{p-1}k^{1-1/p}\qquad (k\to\infty).
\]
Thus
\[
 \left(\frac{S_k}{2k}\right)^p
 \sim 2^{-p}\left(\frac p{p-1}\right)^p\frac1k,
\]
and similarly
\[
 \left(\frac{S_k}{2k+1}\right)^p
 \sim 2^{-p}\left(\frac p{p-1}\right)^p\frac1k.
\]
Combining the two parts, we obtain
\[
 \liminf_{N\to\infty}
 \frac{\displaystyle\int_0^\infty |M_N(t)-A_N(t)|^p\,dt}
 {\displaystyle\int_0^\infty f_N(t)^p\,dt}
 \geq 2^{1-p}\left(\frac p{p-1}\right)^p.
\]
Together with the upper bound proved in Theorem 1, this proves sharpness.
\end{proof}

\section{The discrete form}

The discrete counterpart of Theorem 1 is as follows. Let $p>1$ and let $x_1,x_2,\ldots,x_n\geq0$. Define
\[
 a_i=\frac{x_1+x_2+\cdots+x_i}{i},
\]
and let $m_i$ be the lower median of $x_1,x_2,\ldots,x_i$. Then
\[
 \sum_{i=1}^n |m_i-a_i|^p
 \leq 2^{1-p}\left(\frac p{p-1}\right)^p\sum_{i=1}^n x_i^p. \tag{5}
\]

We give the details. For each $i$, rearrange $x_1,x_2,\ldots,x_i$ in decreasing order:
\[
 y_1^{(i)}\geq y_2^{(i)}\geq\cdots\geq y_i^{(i)}\geq0.
\]
The same half-measure argument as in Lemma 1 gives
\[
 |m_i-a_i|\leq
 \frac{y_1^{(i)}+y_2^{(i)}+\cdots+y_{\lfloor i/2\rfloor}^{(i)}}{i}. \tag{6}
\]
Now rearrange all $x_1,x_2,\ldots,x_n$ in decreasing order:
\[
 x_1^*\geq x_2^*\geq\cdots\geq x_n^*\geq0.
\]
Since the sum of the largest $r$ terms among any initial $i$ terms is no larger than the sum of the largest $r$ terms among all $n$ terms, (6) implies
\[
 |m_i-a_i|\leq \frac1i\sum_{j=1}^{\lfloor i/2\rfloor}x_j^*. \tag{7}
\]
Consequently,
\[
 \sum_{i=1}^n |m_i-a_i|^p
 \leq \sum_{i=1}^n\left(\frac1i\sum_{j=1}^{\lfloor i/2\rfloor}x_j^*\right)^p.
\]
Grouping even and odd indices, and adding a harmless nonnegative term if the last odd term is absent, we have for $r\geq1$,
\[
 \left(\frac1{2r}\sum_{j=1}^r x_j^*\right)^p
 +\left(\frac1{2r+1}\sum_{j=1}^r x_j^*\right)^p
 \leq 2^{1-p}\left(\frac1r\sum_{j=1}^r x_j^*\right)^p.
\]
By the discrete Hardy inequality
\[
 \sum_{r=1}^n\left(\frac1r\sum_{j=1}^r b_j\right)^p
 \leq \left(\frac p{p-1}\right)^p\sum_{r=1}^n b_r^p\qquad (b_r\geq0),
\]
we get
\[
 \sum_{i=1}^n |m_i-a_i|^p
 \leq 2^{1-p}\left(\frac p{p-1}\right)^p\sum_{r=1}^n (x_r^*)^p
 =2^{1-p}\left(\frac p{p-1}\right)^p\sum_{r=1}^n x_r^p.
\]
This proves (5).

In particular, for $p=2$,
\[
 \sum_{i=1}^n (m_i-a_i)^2\leq 2\sum_{i=1}^n x_i^2,
\]
which is the conclusion of the original competition problem.

Finally, we show that the discrete constant also cannot be improved. Let
\[
 x_{2k-1}=0,
 \qquad
 x_{2k}=k^{-1/p},
 \qquad k=1,2,\ldots,N.
\]
For every initial segment, at least half of the terms are zero, and hence the lower median can be taken to be $m_i=0$. Put
\[
 S_k=\sum_{j=1}^k j^{-1/p}.
\]
Then
\[
 a_{2k-1}=\frac{S_{k-1}}{2k-1},
 \qquad
 a_{2k}=\frac{S_k}{2k}.
\]
Using
\[
 S_k\sim \frac p{p-1}k^{1-1/p},
\]
we obtain
\[
 \frac{\displaystyle\sum_{i=1}^{2N}|m_i-a_i|^p}
 {\displaystyle\sum_{i=1}^{2N}x_i^p}
 \longrightarrow
 2^{1-p}\left(\frac p{p-1}\right)^p.
\]
Therefore the constant in the discrete inequality is also sharp.

\end{document}